\documentclass[11pt]{amsart}

\usepackage[utf8]{inputenc}
\usepackage{amsfonts}

\usepackage{latexsym}
\usepackage{amssymb}
\usepackage{amsmath}
\usepackage{mathrsfs}
\usepackage{booktabs}
\usepackage{comment,color}

\usepackage{hyperref}

\numberwithin{equation}{section}

\newtheorem{theorem}{Theorem}[section]
\newtheorem{lemma}[theorem]{Lemma}
\newtheorem{proposition}[theorem]{Proposition}
\newtheorem{corollary}[theorem]{Corollary}
\theoremstyle{definition}
\newtheorem{definition}[theorem]{Definition}

\theoremstyle{remark}
\newtheorem{remark}[theorem]{Remark}

\newcommand{\hR}{\,\!^\ast\R}

\newcommand{\FI}{\mathrm{FI}}
\newcommand{\M}{\mathrm{M}}
\newcommand{\I}{\mathrm{I}}
\newcommand{\NN}{\mathrm{N}}
\newcommand{\AAA}{\mathrm{A}}

\newcommand{\R}{\mathbb{R}}

\newcommand{\Q}{\mathbb{Q}}
\newcommand{\E}{\mathbb{E}}
\newcommand{\F}{\mathbb{F}}
\newcommand{\abf}{\mathbf{a}}
\newcommand{\infinitesimals}{\oslash}
\newcommand{\finite}{\pounds}
\newcommand{\infinite}{\finite^c}

\newcommand{\modelR}{\widehat\R}

\begin{document}

\title{Flexible involutive meadows}
  \author{Emanuele Bottazzi and Bruno Dinis}

\maketitle

\begin{abstract}
We investigate a notion of inverse for neutrices inspired by Van den Berg and Koudjeti's decomposition of a neutrix as the product of a real number and an idempotent neutrix. We end up with an algebraic structure that can be characterized axiomatically and generalizes involutive meadows. The latter are algebraic structures where the inverse for multiplication is a total operation. 
 As it turns out, the structures satisfying the axioms of flexible involutive meadows are of interest beyond nonstandard analysis.
\end{abstract}

\section{Introduction}

Neutrices and external numbers (which can be seen as translations of neutrices over the hyperreal line) were introduced by Van den Berg and Koudjeti in \cite{KoudjetiBerg(95)} as models of uncertainties, in the context of nonstandard analysis,  and further developed in \cite{koudjetithese,IBerg(10),DinisBerg(11),DinisBerg(ta),DinisBerg(book)}. Neutrices were named after and inspired by Van der Corput's groups of functions \cite{Corput(60)} in an attempt to give a mathematically rigorous formulation to the art of neglecting small quantities -- \emph{ars negligendi}.

One of the long-standing open questions in the theory of external numbers is the definition of a suitable notion of inverse of a neutrix. For zeroless external numbers, that is, external numbers that don't contain $0$ and therefore cannot be reduced to a neutrix, there is a notion of inverse, defined
from the Minkowski product between sets (see Definition \ref{def quotient} below), but this cannot be generalized to neutrices.

A meadow (see Sections \ref{S:involutive meadows} and \ref{S:common meadows} below for further details) is a commutative ring with a multiplicative identity element and a total multiplicative inverse operation. Two of the main classes of meadows are involutive meadows, where the inverse of zero is defined to be zero, and common meadows, that instead introduce an error term that propagates through calculations \cite{Bergstra2015}.

One of the motivations for the study of structures where the inverse of zero is defined comes from equational theories \cite{10.1145/1219092.1219095,Komori1975FreeAO,10.2969/jmsj/03520289}. For instance, Ono and Komori introduced such structures motivated from the algebraic study of equational theories and universal theories of fields, and free algebras over all fields, respectively. A long-standing result by Birkhoff states that algebraic structures with an equational axiomatization -- namely, whose axioms only involve equality, besides the functions and constants of the structure itself -- are closed under substructures. Algebraic structures where the inverse is defined only for nonzero elements are not equational, since they have to use inequalities or quantifiers in their definition of a multiplicative inverse. Instead, involutive meadows and common meadows, that define the inverse of zero as zero or a new error term, respectively, admit equational axiomatizations.

Equational axiomatizations of meadows based on known algebraic structures, such as $\Q$ and $\R$, are also of interest to computer science. According to Bergstra and Tucker \cite{10.1145/1219092.1219095}, such equational axiomatizations allow for simple term rewriting systems and are easier to automate in formal reasoning.

Another motivation for the study of meadows is a philosophical
interest in the definition of an inverse of zero (see e.g. \cite[Section 3]{BERGSTRA2011203}). If one wants to assign a meaning to expressions such as $0^{-1}$ or $1/0$ (Bergstra and Middleburg argue that, in principle, these two operations need to be distinguished \cite{BERGSTRA2011203}), the theory of meadows allows for two main options: $(i)$ involutive meadows which define $0^{-1}=0$, resulting into an equational theory closer to that of the original structure, $(ii)$ common meadows which define $0^{-1}$ as a new error term that propagates through calculations.


It turns out that external numbers are particularly suitable for expressing the kind of concepts involved in the definition of the inverse of zero. The key insight is that, being convex subgroups of the hyperreal numbers (i.e.\ the extension of the real number system which includes nonstandard elements such as infinitesimals), neutrices are ``error'' terms, which can also be seen as generalized zeros and are therefore suitable to build models of meadows. The fact that one is using hyperreals (or other non-archimedean field extensions of the real numbers) is crucial, because the real numbers only have two convex subgroups: $\{0\}$ and $\R$, while in the context of the hyperreals there are countably infinitely many, e.g.\ the set of all infinitesimals -- denoted $\oslash$  -- and the set of all limited numbers -- denoted $\pounds$ (see the examples after Definition~\ref{def neutrix and inf}). In turn, external numbers are of the form $a+A$, where $a$ is a hyperreal number and $A$ is a neutrix and can therefore be seen as translations of neutrices. According to the interpretation of neutrices as error terms or generalized zeroes, external numbers can be interpreted as expressing a quantity with a degree of uncertainty.

By introducing an alternative way to define the inverse of a neutrix, inspired by a result by Van den Berg and Koudjeti \cite{KoudjetiBerg(95)} stating that every neutrix can be decomposed as the product of an hyperreal number and an idempotent neutrix, we end up with an algebraic structure that can be characterized axiomatically and generalizes involutive meadows.  Since the new class of structures involves error terms, we call it \emph{flexible involutive meadow}, in the spirit of \cite{JustinoBerg(13)}.

\section{Preliminary notions}

In this section we recall the axioms of common meadows and involutive meadows, with some comments, as well as some background on the external numbers.

\subsection{Involutive meadows}\label{S:involutive meadows}
The axioms of involutive meadows are listed in Figure~\ref{tab:ax_involutive_meadows}.

\begin{figure}[h!]
	\begin{tabular}{l c c r}
		\toprule
		\\[-2mm]
		$(\I_1)$ &\qquad \qquad \qquad&$(x+y)+z=x+(y+z) $  & \qquad \qquad \\[2mm]
		$(\I_2)$ && $x+y=y+x $ \\[2mm]
		$(\I_3)$ && $x+0=x$  \\[2mm]
		$(\I_4)$ && $x+ (-x)=0$ \\[2mm]
		$(\I_5)$ && $(x \cdot y) \cdot z=x \cdot (y \cdot z)$  \\[2mm]
		$(\I_6)$ && $x \cdot y=y \cdot x $ \\[2mm]
		$(\I_7)$ && $1 \cdot x=x$ \\[2mm]
		$(\I_8)$ && $x \cdot (y+z)= x \cdot y + x \cdot z$ \\[2mm]
		$(\I_9)$ && $(x^{-1})^{-1}=x$ \\[2mm]
		$(\I_{10})$ && $x \cdot (x \cdot x^{-1})=x$\\[2mm]
		\toprule
	\end{tabular}
	\caption{Axioms for involutive meadows}
	\label{tab:ax_involutive_meadows}
\end{figure}

The term \emph{involutive} refers to the fact that taking inverses is an involution, as postulated by axiom $(\I_9)$.
With the exception of axiom $(\I_{10})$, the remaining axioms are quite standard, as they postulate the existence of operations of addition $+$ and multiplication $\cdot$ which are associative, commutative, admit a neutral element (denoted $0$ and $1$ respectively). Furthermore, there is an inverse for addition and multiplication is distributive with respect to addition. 

Axiom $(\I_{10})$ replaces the more usual $x \cdot x^{-1} = 1$, which is false for $x = 0$.
This hints at the fact that one should not take the definition of $a^{-1}$ as the number satisfying $a \cdot a^{-1} = 1$. This also ties in with rejecting division as the ``inverse'' of multiplication. 

\subsection{Common meadows}\label{S:common meadows}
The axioms of common meadows are listed in Figure \ref{tab:abc} (see also  \cite{Bergstra2015}).

\begin{figure}[h!]
\begin{tabular}{l c c r}
\toprule
\\[-2mm]
$(\M_1)$ &\qquad \qquad \qquad&$(x+y)+z=x+(y+z) $  & \qquad \qquad \\[2mm]
$(\M_2)$ && $x+y=y+x $ \\[2mm]
$(\M_3)$ && $x+0=x$  \\[2mm]
$(\M_4)$ && $x+ (-x)=0 \cdot x$ \\[2mm]
$(\M_5)$ && $(x \cdot y) \cdot z=x \cdot (y \cdot z)$  \\[2mm]
$(\M_6)$ && $x \cdot y=y \cdot x $ \\[2mm]
$(\M_7)$ && $1 \cdot x=x$ \\[2mm]
$(\M_8)$ && $x \cdot (y+z)= x \cdot y + x \cdot z$ \\[2mm]
$(\M_9)$ && $-(-x)=x$ \\[2mm]
$(\M_{10})$ && $x \cdot x^{-1}=1 + 0 \cdot x^{-1}$\\[2mm]
$(\M_{11})$ && $(x \cdot y)^{-1} = x^{-1} \cdot y^{-1}$\\[2mm]
$(\M_{12})$ && $(1 + 0 \cdot x)^{-1} = 1 + 0 \cdot x $\\[2mm]
$(\M_{13})$ && $ 0^{-1}=\abf$\\[2mm]
$(\M_{14})$ && $x+ \abf = \abf $\\[2mm]
\toprule
\end{tabular}
\caption{Axioms for common meadows}
\label{tab:abc}
\end{figure}

As with involutive meadows, some of the axioms are quite standard (namely $(\M_1)-(\M_3)$, $(\M_5)-(\M_7)$, $(\M_8)$, $(\M_9)$, and $(\M_{11})$), as they postulate the existence of operations of addition $+$ and multiplication $\cdot$ which are associative, commutative and admit a neutral element (denoted $0$ and $1$ respectively). Furthermore, there is an inverse for addition, multiplication is distributive with respect to addition and the inverse of the product of two elements is the product of the inverses.

Axiom $(\M_4)$ postulates the existence of a sort of additive inverse for every element $x$ but with the caveat that the result of operating an element with its inverse is not the neutral element $0$ but $0 \cdot x$.

Axioms $(\M_{10})$ and $(\M_{12})$ concern further properties of the inverse for multiplication. The novelty, compared with more familiar settings, is that they have ``error'' terms in the form of the product of an element $x$, (respectively, its inverse $x^{-1}$) by $0$.

Axiom $(\M_{13})$ defines $0^{-1}$ as an "error term" $\abf$ that does not belong to the initial structure. Due to the presence of this error term, the result of $x \cdot x^{-1}$ is defined as $1 + 0 \cdot x^{-1}$. If $x \ne 0$ and $x \ne \abf$, then $0 \cdot x^{-1} = 0$ (see \cite[Proposition 2.3.1]{Bergstra2015}) and we recover the usual result that holds in a field. If $x = 0$ or $x = \abf$, then the additional term $0 \cdot x^{-1}$ is equal to $\abf$.

Axiom $(\M_{12})$ has a similar motivation: if $x \ne \abf$, then we recover that the inverse of $1$ is $1$. If $x = \abf$, then we get that the inverse of $\abf$ is $\abf$ itself.

\subsection{External numbers}

Let us recall some definitions and results about neutrices and external numbers. We will use  $\hR$ to denote an extension of the real number system that includes nonstandard elements -- such as infinitesimals -- \cite{Goldblatt(98)}, and $\R$ to denote the usual set of real numbers.\footnote{Note that this notation differs from the usual presentations of external numbers, according to which $\R$ already contains nonstandard elements.}

\begin{definition}[Neutrices]
	\label{def neutrix and inf}
	A \emph{neutrix} is a convex subset of $\hR$ that is a subgroup for addition.

\end{definition}

	Some simple examples of neutrices are:
	\begin{itemize}
		\item $\infinitesimals = \{x \in \hR : x \simeq 0\}$;
		\item $\finite = \{x \in \hR : x \text{ is finite}\}$;
		\item if $\varepsilon \simeq 0$, $\varepsilon \finite =  \left\{x \in \hR : \frac{x}{\varepsilon} \text{ is finite}\right\}$.
	\end{itemize}

We are using the notation $x \simeq 0$ to say that $x$ is \emph{infinitesimal}. We will also write $x \simeq y$, and say that $x$ is \emph{infinitely close} to $y$, if $x-y \simeq 0$.

\begin{definition}[External numbers]\label{def ext numbers}
	An \emph{external number} $\alpha$ is the sum of a hyperreal number $a$ and a neutrix $A$ in the following sense:
	$$
		\alpha = a + A = \{ a+r : r\in A \}.
	$$
	
	An external number $\alpha = a + A$ that is not reduced to a neutrix (equivalently, such that $0 \not \in \alpha$) is said to be \emph{zeroless}.
\end{definition}

The sum and product of external numbers 
is introduced in the following definition. We refer to \cite{DinisBerg(book)} for their properties. We will however recall some of them which will be useful later on.

\begin{definition}\label{def minkowski}
	For $a, b\in \hR$ and $A, B \subseteq \hR$ (not necessarily neutrices), we define the Minkowski sum and product
	\begin{align*}
		(a+A)+(b+B) & = (a+b)+(A+B)\\
		(a+A)\cdot(b+B) & = ab+ aB+ bA +AB,
	\end{align*}
	where
	\begin{align*}
		A+B & = \{x+y : x\in A \land y \in B\}\\
		aB & = \{ay :  y \in B\}\\
		AB & = \{xy : x\in A \land y \in B\}.
	\end{align*}
\end{definition}

It is also possible to define a notion of division between subsets of hyperreal numbers, even if they contain $0$.
 
\begin{definition}\label{def quotient}
	For $A, B \subseteq \hR$ (not necessarily neutrices), we define
	$$
		{A}\cdot{B}^{-1} = \{ x : xB \subseteq A \}.
	$$
\end{definition}

Usually, the operation ${A}\cdot{B}^{-1}$ is written as $\frac{A}{B}$. Here we chose the inverse notation since we are investigating structures related to meadows, whose axioms are commonly stated in terms of the inverse operation. For further discussion on the use of these operations, see \cite{BERGSTRA2011203}.

Notice that Definition \ref{def quotient} doesn't allow us to obtain a \emph{proper} inverse of a neutrix. In fact, if $A = \{1\}$ and $B$ is a neutrix, ${A}\cdot{B}^{-1}$ is empty, since for no $x$ we have $0x\in A$. This example motivated us to look for alternative definitions of inverses of a neutrix.


If $x=a+A$ is zeroless, then ${a}^{-1}\cdot{A} \subseteq A$, and therefore ${a}^{-1}\cdot{A}+A=A$. In fact, ${a}^{-1}\cdot{A}\subseteq \infinitesimals$.


\begin{proposition}[{\cite[Theorem 1.4.2]{DinisBerg(book)}}]\label{prop inverse}
	Let $x = a + A$ be a zeroless external number. Then
	$${x}^{-1} = (a+A)^{-1} = {a}^{-1}+{a}^{-2}\cdot{A}.$$
\end{proposition}


\begin{proposition}\label{P:extnumbersFIM}
Let $x,y,z \in \E$, be such that $x=a+A$. Then
\begin{enumerate}
\item $x+(y+z)=(x+y)+z$
\item $x+y=y+x $
\item $x+A=x$
\item $x+ (-x)=A$
\item $(x \cdot y) \cdot z=x \cdot (y \cdot z)$
\item $x \cdot y=y \cdot x $
\item $x \cdot (y+z)= x \cdot y + x \cdot z + A\cdot y + A \cdot z$
\item $(x^{-1})^{-1}=x$.
\end{enumerate}
\end{proposition}

\section{Flexible involutive meadows}\label{S:flexible}

A neutrix $I$ is said to be \emph{idempotent} if $I \cdot I=I$. As showed by Van den Berg and Koudjeti in \cite{KoudjetiBerg(95)} (see also \cite{dinisberg2015}) every neutrix is a multiple of an idempotent neutrix. 

\begin{theorem}[{\cite[Theorem~7.4.2]{KoudjetiBerg(95)}}]\label{T:N is a real times an idempotent}
Let $N$ be a neutrix. Then, there exists a hyperreal number $r$ and a unique idempotent neutrix $I$ such that $N= r\cdot I$.
\end{theorem}

We use the previous result to define inverses for neutrices.


\begin{definition}\label{D:inverse}
Let $x = a+A$, where $A = r \cdot I$, for some hyperreal number $r$ and idempotent neutrix $I$. We define the \emph{inverse} of $x$, denoted $x^{-1}$ as follows:
\begin{equation*}\label{equation def inverse with Imme's result}
	x^{-1} = \begin{cases}
		{a}^{-1} + {a}^{-2}\cdot{A}& \text{, if } a \ne 0\\
		{r}^{-1} \cdot I & \text{, otherwise}.
	\end{cases}
\end{equation*}
\end{definition}

In the decomposition $N = r \cdot I$ of Theorem \ref{T:N is a real times an idempotent}, the idempotent neutrix $I$ is uniquely determined, but the number $r$ is not. Nevertheless, the inverse given in Definition \ref{D:inverse} is uniquely defined, as a consequence of the next proposition. 

\begin{proposition}
If $r \ne s$ satisfy $N = r \cdot I = s \cdot I$, then also ${r}^{-1} \cdot I = {s}^{-1} \cdot I$.
\end{proposition}

\begin{proof}
We may assume, without loss of generality that $0<r <s$, which implies that ${s}^{-1} < {r}^{-1}$. Suppose towards a contradiction that ${r}^{-1} \cdot I \ne {s}^{-1} \cdot I$. By our assumptions over $r$ and $s$, this implies ${s}^{-1}\cdot I \subsetneq {r}^{-1} \cdot I$. Then, there exists some $i \in I$ such that
$i\cdot r^{-1} \not \in s^{-1}\cdot I$.
If we multiply by $r$, we obtain
$$
	i \not \in s^{-1} \cdot (r \cdot I) = s^{-1} \cdot (s \cdot I) = I,
$$
which contradicts the assumption that $i \in I$. Hence ${r}^{-1} \cdot I = {s}^{-1} \cdot I$.
\end{proof}

Alternatively, since there are countably many non-isomorphic neutrices \cite{IBerg(10)}, an application of the Axiom of Countable Choice would allow us to choose a canonical representative for the element $r$ in the decomposition of Theorem~\ref{T:N is a real times an idempotent}. This is sufficient to uniquely determine $r$ and to also guarantee that Definition~\ref{D:inverse} is well-posed.

We show that the external numbers equipped with the inverse defined in Definition~\ref{D:inverse} satisfy the axioms given in Figure~\ref{tab:ax_flexible_involutive_meadows}, where $N(x)$ denotes the neutrix part of the external number $x$. As such, one can also think of $N(x)$ as an error term, or a generalized zero, such that every $x$ decomposes uniquely as $x = r + N(x)$ with $N(r) = 0$. We call any structure satisfying the axioms in Figure~\ref{tab:ax_flexible_involutive_meadows} a \emph{flexible involutive meadow}.


\begin{figure}[h!]
	\begin{tabular}{l c c r}
		\toprule
		\\[-2mm]
		$(\FI_1)$ &\qquad \qquad \qquad&$(x+y)+z=x+(y+z) $  & \qquad \qquad \\[2mm]
		$(\FI_2)$ && $x+y=y+x $ \\[2mm]
		$(\FI_3)$ && $x+N(x)=x$  \\[2mm]
		$(\FI_4)$ && $x+ (-x)=N(x)$ \\[2mm]
		$(\FI_5)$ && $(x \cdot y) \cdot z=x \cdot (y \cdot z)$  \\[2mm]
		$(\FI_6)$ && $x \cdot y=y \cdot x $ \\[2mm]
		$(\FI_7)$ && $\left(1+{N(x)}\cdot{x^{-1}}\right) \cdot x=x$ \\[2mm]
		$(\FI_8)$ && $x \cdot (y+z)= x \cdot y + x \cdot z + N(x)\cdot y + N(x) \cdot z$ \\[2mm]
		$(\FI_9)$ && $(x^{-1})^{-1}=x$ \\[2mm]
		$(\FI_{10})$ && $x \cdot (x \cdot x^{-1})=x$\\[2mm]
		\toprule
	\end{tabular}
	\caption{Axioms for flexible involutive meadows}
	\label{tab:ax_flexible_involutive_meadows}
\end{figure}

These axioms provide a generalization of the axioms of involutive meadows given in Figure~\ref{tab:ax_involutive_meadows}.
The difference between involutive meadows and flexible involutive meadows is that we replace $0$ by generalized zeroes $N(x)$; $1$ by a generalized $1$ (i.e.\ $1$ plus an error term of the form $N(x)$); and distributivity by a generalized form of distributivity which holds for external numbers. Notice that, in the context of the external numbers, the generalized zeroes take the form of neutrices. Moreover, $N(x)\cdot y + N(x) \cdot z$ is a neutrix, so the error term in the generalized distributivity axiom is once again a neutrix. In fact, if one interprets $N(x)$ as being $0$, for all $x$, then one recovers the axioms for involutive meadows.

We now want to prove that the external numbers with the usual addition and multiplication and with the inverse defined by \eqref{equation def inverse with Imme's result} satisfy the axioms for flexible involutive meadows. In order to prove this result, we will use the following properties of the inverse of a neutrix.

\begin{lemma}\label{lemma for the main theorem}
	Let $x = N(x)=r \cdot I$, with $I$ an idempotent neutrix. Then:
	\begin{enumerate}
		\item $(1+x\cdot x^{-1} ) \cdot x  = x$; 
		\item $(x^{-1})^{-1} = x$; 
		\item $x\cdot(x\cdot x^{-1}) = x$ 
	\end{enumerate}
\end{lemma}
\begin{proof}
	\begin{enumerate}
	\item We have that $x\cdot x^{-1}=N(x) \cdot x^{-1} = I$. Hence $$(1+x\cdot x^{-1} ) \cdot x=(1+I) \cdot (r \cdot I) = r\cdot I + r \cdot I^2 = r \cdot I = x.$$
	\item We have 
	\[
	(x^{-1})^{-1} = (r^{-1} I)^{-1} = (r^{-1})^{-1} I = rI = x.
	\]
	\item We have
	
	\[x\cdot(x\cdot x^{-1}) = rI \cdot (rI \cdot r^{-1} I) = rI \cdot I = rI = x. \qedhere\]
	\end{enumerate}
\end{proof}

As proved below, the external numbers satisfy an additional property related to the following \emph{Inverse Law} of involutive meadows:
$$
	x\ne 0 \Rightarrow x\cdot x^{-1} = 1
$$
Involutive meadows that satisfy the Inverse Law are called \emph{cancellation meadows} and are of particular interest. In fact, in \cite{sacscuza:bergstra2017sotvom} it is proved that every involutive meadow is a subdirect product of cancellation meadows.

In the setting of flexible involutive meadows, the inverse law is more suitably expressed by its flexible counterpart:

\begin{equation}\label{e:FIL}
x \ne N(x) \Rightarrow x\cdot x^{-1} = 1 + e,
\end{equation}
where $e$ is an error term such that $1 + e$ is not an error term.

Moreover, by part 3 and part 4 of Proposition~\ref{P:extnumbersFIM} the external numbers satisfy the following arithmetical properties that generalize the  properties of arithmetical meadows \cite{BERGSTRA2011203} 
\begin{equation*}
	\begin{split}
		(\AAA_1) & \qquad   x+(-x) =N(x)\\
		(\AAA_2) &  \qquad x+N(x)=x.
	\end{split}
\end{equation*}

\begin{theorem}\label{main theorem}
	The external numbers with the usual addition and multiplication and with the inverse introduced in Definition \ref{D:inverse} satisfy the axioms for flexible involutive meadows plus the Flexible Inverse Law given by \eqref{e:FIL} and the arithmetical properties $(\AAA_1)$ and $(\AAA_2)$.
\end{theorem}
\begin{proof}
	By Proposition~\ref{P:extnumbersFIM}, in order to show that the external numbers are a flexible involutive meadow we only need to verify $(\FI_7)$ and $(\FI_{10})$. If $x$ is a neutrix, both axioms hold due to Lemma~\ref{lemma for the main theorem}. Assume that $x=a+A$ is zeroless. Then, using the algebraic properties of the external numbers one derives 
	\begin{equation*}
	\begin{split}
(1+N(x)\cdot x^{-1})\cdot x &=\left(1+A\left({a}^{-1}+{a}^{-2}\cdot{A}\right)\right)(a+A)\\
&=\left(1+\left({a}^{-1}\cdot A+{a^{-2}}\cdot{A^2}\right)\right)(a+A)\\
&=a+A+{a}^{-1}\cdot{A^2}+A+{a}^{-1}\cdot{A^3}\\
&=a+A=x.
	\end{split}
	\end{equation*}

	Hence $(\FI_7)$ holds. As regarding $(\FI_{10})$ one has
	
	\[
	x(x\cdot x^{-1})=(a+A)\left(1+{a}^{-1}\cdot {A}\right)=a+A+A+{a}^{-1}\cdot{A^2}=a+A=x.
	\]
	Hence $(\FI_{10})$ also holds and therefore the external numbers are a flexible involutive meadow.
	
	We now show the Flexible Inverse Law. Let $x=a+A$ be a zeroless external number. Then 
	$$
		x \cdot x^{-1} = 1 + a^{-1}\cdot A + a^{-1}\cdot A+{a^{-2}}\cdot{A^2} = 
		1 + a^{-1}\cdot A.
	$$
	Since $x$ is zeroless, $a^{-1}\cdot A \subseteq \infinitesimals$, so the Flexible Inverse Law is satisfied.
\end{proof}

\begin{corollary}
	The axioms for flexible involutive meadows are consistent.
\end{corollary}

\subsection{Some properties of flexible involutive meadows}
We now prove some basic properties of flexible involutive meadows. We start by showing an additive cancellation law and that $N(\cdot)$ is idempotent for addition.

\begin{proposition}\label{P:properties1}
Let $M$ be a flexible involutive meadow and
let $x,y,z\in M$. Then 
\begin{enumerate}
\item $x+y=x+z$ if and only if $N(x)+y=N(x)+z.$
\item $N(x)+N(x)=N(x)$.
\end{enumerate}
\end{proposition}

\begin{proof}
\begin{enumerate}
\item Suppose firstly that $x+y=x+z$. Then%
\begin{equation*}
N(x)+y=-x+x+y=-x+x+z=N(x)+z.
\end{equation*}

Suppose secondly that $N(x)+y=N(x)+z$. Then%
\begin{equation*}
x+y=x+N(x)+y=x+N(x)+z=x+z.
\end{equation*}

\item This follows from applying part 1 to axiom $(\FI_{3})$. \qedhere
\end{enumerate}
\end{proof}

In order to prove other basic properties that allow to operate with the $N(\cdot)$ function and with additive inverses one needs to assume the following two extra axioms

\begin{equation*}
\begin{split}
(\NN_1) & \qquad   N(x+y)=N(x) \vee N(x+y)=N(y)\\
(\NN_2) &  \qquad N(-x)=N(x)
\end{split}
\end{equation*}

\begin{proposition}\label{P:properties2}
Let $M$ be a flexible involutive meadow satisfying also $(\NN_1)$ and $(\NN_2)$, and
let $x,y,z\in M$. Then 
\begin{enumerate}
\item $N(x+y)=N(x)+N(y)$.  
\item $N(N(x))=N(x)$.
\item If $x=N\left( y\right) $, then $x=N\left( x\right) $.
\item $-(-x)=x$.
\item $-(x+y)=-x-y$. 
\item $N(x)=-N(x)$.
\end{enumerate}
\end{proposition}

\begin{proof}
\begin{enumerate}
\item One has 
\begin{equation*}
x+y+N(x)+N(y)=x+N(x)+y+N(y)=x+y.
\end{equation*}%
Then by part 1%
\begin{equation}\label{appcl}
N(x+y)+N(x)+N(y)=N(x+y).  
\end{equation}%

By $(\NN_1)$  one has $N(x+y)=N(x)$ or $N(x+y)=N(y)$. Suppose
that $N(x+y)=N(x)$. Then by \eqref{appcl} and Proposition~\ref{P:properties1}, 
\begin{equation*}
N(x+y)=N(x)+N(x)+N(y)=N(x)+N(y).
\end{equation*}%
If $N(x+y)=N(y)$ the proof is analogous.

\item Using Proposition~\ref{P:properties1} and part 1 we have
\[
N(N(x))=N(x-x)=N(x)+N(-x)=N(x)+N(x)=N(x).
\]

\item Using part 2 we derive that $N(x)=N(N(y))=N(y)=x$.
\item We have that 
\begin{equation*}
N(-(-x))=N(-x)=N(x)=-x+x.
\end{equation*}%
Hence%
\[
-(-x) =-(-x)+N(-(-x))=-(-x)-x+x =N(-x)+x=N(x)+x=x.
\]

\item By part 1
\[
-(x+y)+x+y =N(x+y)=N(x)+N(y)=-x+x-y+y=-x-y+x+y.
\]%
Then by Proposition~\ref{P:properties1}
\begin{equation*}
-(x+y)+N(x+y)=-x-y+N(x+y).
\end{equation*}%
Again using part 1 one obtains 
\[
-(x+y)+N(-(x+y)) =-x-y+N(-x)+N(-y) =-x+N(-x)-y+N(-y).
\]

Hence $-(x+y)=-x-y.$

\item 
By part 4 we have
\[
N(x)=-x+x=-x-(-x)=-(x-x)=-N(x). \qedhere
\] 
\end{enumerate}
\end{proof}

\subsection{Involutive flexible meadows are varieties}
In \cite{sacscuza:bergstra2017sotvom}, Bergstra and Bethke studied the relations between involutive meadows and varieties. One of their results is that involutive meadows are varieties. We prove that involutive flexible meadows are also varieties.

Let us start by recalling the definition of varieties in this context, following \cite{universal_algebra}.

\begin{definition}\label{def algebra}
	If $\mathcal{F}$ is a signature, then an \emph{algebra} $A$ of type $\mathcal{F}$ is defined as an ordered pair $(A, F)$, where $A$ is a nonempty set and $F$ is a family of finitary operations on $A$ in the language of $\mathcal{F}$ such that, for each $n$-ary function symbol $f$ in $\mathcal{F}$, there is an $n$-ary operation $f^A$ on $A$. 
\end{definition}

\begin{definition}\label{def variety}
	A nonempty class $K$ of algebras of the same signature is called a \emph{variety} if it is closed under subalgebras, homomorphic images, and direct products. 
\end{definition}

A result by Birkhoff entails that $K$ is a variety if and only if it can be axiomatized by identities.

\begin{definition}\label{def equational class}
	Let $\Sigma$ be a set of identities over the signature $\mathcal{F}$; and define $M(\Sigma)$ to be the class of algebras $A$ satisfying $\Sigma$. A class $K$ of algebras is an \emph{equational class} if there is a set of identities $\Sigma$ such that $K = M(\Sigma)$. In this case we say that $K$ is defined, or axiomatized, by $\Sigma$.
\end{definition}

\begin{theorem}\label{birkhoff's theorem}
	$K$ is a variety if and only if it is an equational class.
\end{theorem}

The class $K$ of flexible involutive meadows is axiomatized by the identities in Figure~\ref{tab:ax_flexible_involutive_meadows} over the signature $\Sigma = \{+, \cdot, -, \mbox{}^{-1}, 0, 1, N(\cdot) \}$, where $N$ is a unary function that, when interpreted with the external numbers, corresponds to the neutrix part of a number $x$.
As a consequence of Birkhoff's theorem, we have the following result.

\begin{corollary}\label{corollary varieties}
	Involutive flexible meadows are varieties.
\end{corollary}

\subsection{Involutive flexible meadows and commutative von Neumann regular rings}
In the investigation of meadows, the relations with commutative von Neumann regular rings with a multiplicative identity element seem to be of particular interest \cite{BERGSTRA20091261,BERGSTRA2011203}. 

We recall that a semigroup $(S,\cdot)$ is said to be \emph{Von Neumann regular} if

\[
\forall x \in S \, \exists y \in S \left(x \cdot x \cdot y =x \right).
\]

A commutative von Neumann regular ring with a multiplicative identity is a Von Neumann regular commutative semigroup for both addition and multiplication.


Involutive flexible meadows are also commutative von Neumann regular rings with a multiplicative identity element.

\begin{proposition}
Let $M$ be an involutive flexible meadow. Then $M$ is a Von Neumann regular commutative semigroup for both addition and multiplication.
\end{proposition}

\begin{proof}
This is a simple consequence of associativity together with axioms $(\FI_{4})$ and $(\FI_{10})$.
\end{proof}

In \cite[Lemma 2.11]{BERGSTRA20091261} it was shown that commutative von Neumann regular rings can be expanded in a unique way to an involutive meadows. Since involutive meadows are also flexible involutive meadows, von Neumann regular rings can be expanded to flexible involutive meadows. The expansion to flexible involutive meadows might not be unique, though, due to the presence of different error terms.

Further research on the connection between commutative von Neumann regular rings and flexible involutive meadows goes beyond the scope of this paper.

\subsection{Solids are involutive flexible meadows}

We finish this section by showing that instead of working with the external numbers, one can use a purely algebraic approach by working with a structure called a \emph{solid}. For a full list of the solid axioms we refer to the appendix in \cite{DinisBerg(17)} or \cite{DinisBerg(ta)}.

The following proposition compiles some results from \cite[Propositions 2.12, 4.8 and Theorem 2.16]{DinisBerg(ta)} and \cite[Proposition~4.12]{DinisBerg(11)} which we use to show that solids are flexible involutive meadows.

\begin{proposition}\label{Prop_solids}
Let $S$ be a solid and let $x, y,z \in S$. 
\begin{enumerate}
\item If $x = e(x)$, then $e(x y) = e(x)y$.
\item If $x \neq e(x)$, then $u(x)e(x) = xe(u(x)) = e(x)$.
\item $x (z + e (y)) = xz + xe (y)$.
\item If $x \neq e(x)$, then $d(d(x))=x$.
\end{enumerate}
\end{proposition}

\begin{theorem}
Every solid is a flexible involutive meadow.
\end{theorem}

\begin{proof}
Let $S$ be a solid.  
Most of the axioms of flexible involutive meadows are also axioms of solids, by considering $N(x)=e(x)$, $-x=s(x)$, $x^{-1}=d(x)$, and $1=u$. The only non-obvious cases are the cases of axioms $(\FI_7)$, $(\FI_9)$ and $(\FI_{10})$.

For axiom $(\FI_{7})$, if $x \neq e(x)$, using the solid axioms and Proposition~\ref{Prop_solids} we obtain
\[
(1+N(x)\cdot x^{-1})\cdot x =x+e(x)d(x)x=x+e(x)u(x)=x+e(x)=x.
\]

If $x=e(x)$,  the result follows from Lemma~\ref{lemma for the main theorem}(1). 

Axiom $(\FI_9)$ follows from Lemma~\ref{lemma for the main theorem}(2), if $x=e(x)$ and from Proposition~\ref{Prop_solids}(4).

Finally, axiom $(\FI_{10})$ follows easily from the solid axioms if $x \neq e(x)$ and from Lemma~\ref{lemma for the main theorem}(3) if $x=e(x)$.
\end{proof}


As it turns out, and as mentioned above, one is not forced to work in a nonstandard setting. Indeed, any non-archimedean ordered field yields a model of flexible involutive meadows.

Let $\F$ be a non-archimedean ordered field.
Let $\mathcal{C}$ be the set of all convex subgroups for addition of $\F$ and
$Q$ be the set of all cosets with respect to the elements of $\mathcal{C}$. In \cite{DinisBerg(17)} the elements of $\mathcal{C}$ were called \emph{magnitudes} and $Q$ was called the
\emph{quotient class} of $\F$ with respect to $\mathcal{C}$. In the same paper it was also shown that the quotient class of a non-archimedean field is a solid. Hence we have following corollary.

\begin{corollary}
The quotient class of a non-archimedean ordered field is a flexible involutive meadow.
\end{corollary}

\section{Further models for meadows using external numbers}\label{S:models}

In the introduction we claimed that the external numbers are particularly suitable for expressing the kind of concepts involved in the definition of the inverse of zero. In order to support that claim, we explore further models for meadows inspired by the external numbers. We start by building a model for involutive flexible meadows over a finite field $\F$ and proceed by constructing a model for common meadows over $\hR$.

\subsection{Finite models of flexible involutive meadows}

In this subsection we show that  any finite field can be extended to a finite model of an involutive flexible meadow.

\begin{definition}
	Given a finite field $\F$ we define $\widehat{\F}$ as follows. 
	\begin{itemize}
		\item For every $a \in \F$, we set $\widehat{a}=a+\infinitesimals$ and $\widehat{\F} = \{ \widehat{a} : a \in \F \}$.
		\item For every nonzero $a \in \F$, we set $\widehat{a}^{-1}=\widehat{a^{-1}}$.
		\item $(\widehat{0})^{-1}=\widehat{0}$ (this definition is motivated by, and indeed coincides with the one in Definition \ref{D:inverse}).
		\item The sum and product over $\widehat{\F}$ are the Minkowski operations introduced in Definition \ref{def minkowski}. 
	\end{itemize}
\end{definition}

\begin{lemma}
	The operations in $\widehat{\F}$ are compatible with those in $\F$.
\end{lemma}
\begin{proof}
	Let $a,b \in\F$. Then 
	$$
	\widehat{a+b} = (a + b) + \infinitesimals
	=
	(a+\infinitesimals) + (b+\infinitesimals)
	=
	\widehat{a}+\widehat{b}.
	$$
	Since $a \cdot \infinitesimals = b \cdot \infinitesimals =\infinitesimals$ we also have
	$$
	\widehat{a\cdot b} = a \cdot b + \infinitesimals
	=
	(a+\infinitesimals) \cdot (b+\infinitesimals)
	=
	\widehat{a}\cdot\widehat{b}.
	$$
	
	Moreover, if $a\ne0$, then 
	$\widehat{a}^{-1}=\widehat{a^{-1}} = a^{-1} + \infinitesimals$.	
\end{proof}

\begin{theorem}
	Let $\F$ be a finite field. Then $\widehat{\F}$ satisfies the axioms for flexible involutive meadows.
\end{theorem}
\begin{proof}
	If $x \ne \widehat{0}$, the axioms are satisfied since $\F$ is a field and the operations in $\widehat{\F}$ are compatible with the ones in $\F$.
	
	If $x = \widehat{0}$, axioms $(\I_1) - (\I_8)$ follow from the fact that $\F$ is a field.
	
	As for axiom $(\I_9)$, if $x = \widehat{0}$, then $\widehat{0}^{\ -1} = \widehat{0}$, so that $(\widehat{0}^{\ -1})^{-1} = \widehat{0}^{\ -1} = \widehat{0}$.
	
	Finally, for axiom $(\I_{10})$, if $x = \widehat{0}$, then $\widehat{0} \cdot \widehat{0}^{\ -1} = 0$.
\end{proof}

\begin{remark}
	Similar models for involutive meadows can be obtained without recurring to infinitesimals, as one could define the alternative model $\tilde{\F}$ by requiring the existence of an element $E \notin \F$ and defining, for each $a \in \F$ the element $\tilde{a}=a+E$ and the set  $\tilde{\F}:=\{\tilde{a}: a \in \F\}$ with the operations
	
	\[
	\tilde{a}+\tilde{b}=(a+E)+(b+E)=(a+b)+E,
	\] 
	(note that, in particular $E+E=(0+E)+(0+E)=(0+0)+E=0+E=E$),
	\[
	\tilde{a}\cdot \tilde{b}=(a+E)\cdot (b+E)=(a\cdot b)+E
	\]
	
	and 
	
	\[
	\tilde{a}^{-1}=a^{-1}+E
	\]
	
	and, finally,
	
	$$\tilde{0}^{-1}=0+E=\tilde{0}.$$
	
	By adding $E$ and defining $0^{-1}=E$, the resulting structure is a common meadow. If, alternatively, we were to define $0^{-1}=0$, the resulting structure is an involutive meadow. 
\end{remark}

\subsection{A model for common meadows based on $\R$}\label{sec_model}

In this section we introduce a model $\modelR$ for the axioms of common meadows given in Figure~\ref{tab:abc}. In our model, the elements of $\R$ will be represented by external numbers, while $\hR$ will act as an inverse of (the representative of) $0$. The fact that the inverse of $0$ has, in a sense, the maximum possible uncertainty is in good accord with both  the intuition that division by $0$ introduces an error term, and to the common practice of having the inverse of $0$ not being a member of the original field \cite{Bergstra2015}.

The inverse of $\hR$ is again $\hR$. This choice can be justified in two ways. We can interpret $\hR^{-1}$ as the smallest neutrix collecting all the inverses of the elements of $\hR$ or, alternatively, since $0 \in \hR$, the inverse of $\hR$ should also be maximal. 


\begin{definition}
	We define the set $\modelR$ as follows:
	\begin{itemize}
		\item For every $r \in \R$, we set $\widehat{r}=r+\infinitesimals \in \modelR$.
		\item $\hR \in \modelR$.
		\item For every nonzero $r \in \R$, we set $\widehat{r}^{-1} =\widehat{1}\cdot{\widehat{r}^{-1}}$ (with the quotient introduced in Definition \ref{def quotient}).
		\item We define $(\widehat{0})^{-1}=\hR$ and $\hR^{-1} = \hR$.
		\item The sum and product over $\modelR$ are the Minkowski operations introduced in Definition \ref{def minkowski}. 
	\end{itemize}
\end{definition}

In the construction of real numbers as limits of Cauchy sequences, a real number can be seen as being determined up to ``infinitesimals", i.e. up to a sequence converging to $0$. In the model of common meadows introduced in the previous definition, this idea is expressed by the representation of $r$ as $\hat{r}=r+\infinitesimals$.

An immediate consequence of the previous definition is that for every $x \in \modelR$ we have
\begin{equation}\label{equation sum and product with hR}
	x + (\widehat{0})^{-1} = (\widehat{0})^{-1} + x
	=
	x \cdot (\widehat{0})^{-1} = (\widehat{0})^{-1} \cdot x
	= \hR.
\end{equation}

In the next lemma, we establish that the operations in $\modelR$ are compatible with those in $\R$.

\begin{lemma}\label{lemma operations commute with hat}
	For every $r, s \in \R$,
	\begin{itemize}
		\item $\widehat{r+s}=\widehat{r}+\widehat{s}$;
		\item $\widehat{r\cdot s}=\widehat{r}\cdot\widehat{s}$.
	\end{itemize}
	Moreover, for every nonzero $r \in \R$, $\widehat{r}^{-1} = \widehat{r^{-1}}$.
\end{lemma}
\begin{proof}
	The first two properties are a consequence of the following equalities
	$$
	\widehat{r+s} = r + s + \infinitesimals
	=
	(r+\infinitesimals) + (s+\infinitesimals)
	=
	\widehat{r}+\widehat{s}.
	$$
	and, taking into account that $r$ and $s$ are real numbers,
	$$
	\widehat{r\cdot s} = r \cdot s + \infinitesimals
	=
	(r+\infinitesimals) \cdot (s+\infinitesimals)
	=
	\widehat{r}\cdot\widehat{s}.
	$$
	
	As for the inverse, if $r\ne0$,
	$\widehat{r^{-1}} = \frac{1}{r} + \infinitesimals$, whereas, by Proposition \ref{prop inverse},
	\[
	\widehat{r}^{-1}
	=
	\frac{1+\infinitesimals}{r+\infinitesimals}
	=
	(1+\infinitesimals)\left(\frac{1}{r}+\frac{\infinitesimals}{r^2}\right)
	=
	(1+\infinitesimals)\left(\frac{1}{r} + \infinitesimals\right)
	=
	\frac{1}{r} + \infinitesimals.
	\qedhere \]
\end{proof}

\begin{corollary}\label{corollary that might be useful}
	For every $r \in \R$,
	\begin{enumerate}
		\item if $r\ne 0$, then $\widehat{r}\cdot\widehat{r}^{-1}=\widehat{1}$;
		\item if $r\ne 0$, then $\widehat{0} \cdot \widehat{r}^{-1} = \widehat{0}$;
		\item $\widehat{1}=\widehat{1}+\widehat{0}\cdot\widehat{r}$;
		\item $\widehat{r} + \hR = \hR + \widehat{r} = \hR$ and $\widehat{r} \cdot \hR =  \hR \cdot \widehat{r} = \hR$.
		\item $\hR + \hR = \hR - \hR = \hR \cdot \hR = \hR$.
	\end{enumerate}
\end{corollary}

Equalities (1)--(3) in Corollary \ref{corollary that might be useful} can be obtained from the corresponding equalities for real numbers by repeated use of Lemma \ref{lemma operations commute with hat}.

\begin{theorem}
	$\modelR$ is a model of axioms $(\M_1)$--$(\M_{14})$ of common meadows.
\end{theorem}
\begin{proof}
	We start by showing that axiom $(\M_1)$ is satisfied. If $x, y$ and $z$ are different from $\hR$, then this is a consequence of Lemma \ref{lemma operations commute with hat}. If at least one of  $x, y$ and $z$ is equal to $\hR$, then both sides of the equality evaluate to $\hR$ by Corollary \ref{corollary that might be useful}. The proof follows similar steps for axioms $(\M_{2}) - (\M_9)$. Note that, for axiom $(\M_8)$, we use also the fact that we have only one order of magnitude besides $\hR$.
	
	
	
	
	
	
	
	
	
	Let us show that axiom $(\M_{10})$ is satisfied.  If $x = \widehat{0}$:
	$$\widehat{0} \cdot (\widehat{0})^{-1} = \hR = \widehat{1} + \widehat{0} \cdot \hR.$$
	If $x = \hR$: as a consequence of the definition of $\hR^{-1}$, we have
	$$\hR \cdot \hR^{-1} = \hR = \widehat{1} + \widehat{0} \cdot \hR = \widehat{1} + \widehat{0} \cdot \hR^{-1}.$$
	
	We now turn to axiom $(\M_{11})$. If $x, y$ are not equal to $0$ nor to $\hR$, the axiom holds as a consequence of Lemma \ref{lemma operations commute with hat}. Otherwise, due to the definition of the inverse, both sides are equal to $\hR$.
	
	For axiom $(\M_{12})$, if $x = \widehat{r}$ for some $r \in \R$, the axiom holds as a consequence of Corollary \ref{corollary that might be useful} (3). If $x = \hR$, then
	$$(\widehat{1}+\widehat{0}\cdot \hR)^{-1} = (1+\hR)^{-1} = \hR^{-1} = \hR = \widehat{1} + \hR = \widehat{1} + \widehat{0} \cdot \hR.$$
	
	Axiom $(\M_{13})$ is satisfied as a consequence of the definition of $(\widehat{0})^{-1}$. Finally, axiom $(\M_{14})$ is satisfied as a consequence of \eqref{equation sum and product with hR}.
\end{proof}

\section{Final remarks and open questions}

In this paper we have introduced the notion of flexible involutive meadow, by means of an equational axiomatization, and constructed some models based on the external numbers and non-archimedean fields. We have also shown a model for common meadows based on the real numbers and of involutive meadows based on finite fields. We would like to point out that, with similar techniques, one could also obtain meadows based on rational numbers. 

The model for common meadows developed in Section \ref{sec_model} suggests that it is possible to study a flexible version of common meadows, in the spirit of what has been done in Section \ref{S:flexible} for involutive meadows. In order to do so, it is possible to adapt the axioms by replacing $0$ with $N(x)$, where $N(x)$ is an error term analogous to that of flexible involutive meadows.

In the context of external numbers, where $N(x)$ is the neutrix part of $x$, and for zeroless $x$, one has the inclusion $N(x) \subseteq x \cdot \infinitesimals$. This grounds the interpretation of the flexible counterparts of axioms ($\M_4$), ($\M_{10}$) and ($\M_{12}$).
Moreover, as in the model discussed in Section \ref{sec_model}, the element $\abf$ can be taken as $\hR$. 

To conclude, we mention some possible directions of future work.

We would like to study related variants of meadows as well as their algebraic properties. For example, the study of \emph{flexible} cancellation meadows, i.e.\ meadows in which the multiplicative cancellation axiom $$x\ne 0 \land x\cdot y = x\cdot z \Rightarrow y = z$$ or its flexible counterpart (where we substitute $0$ by an error term $e$) holds; or \emph{flexible} arithmetical meadows in the sense of \cite{BERGSTRA2011203}; or \emph{flexible} meadows of rational numbers (see e.g.\ \cite{10.1145/1219092.1219095}).

Are flexible arithmetical meadows, i.e.\ flexible meadows satisfying $(\AAA_1)$ and $(\AAA_2)$ (necessarily) connected with nonstandard models of arithmetic? As for the flexible meadows of rational numbers, are they a minimal algebra? If so, that might provide a connection with data types. Finally, the study of other algebraic concepts, such as morphisms and ideals which are paramount in many algebraic frameworks, are also natural lines of research in the context of flexible meadows. 

\subsection*{Acknowledgments}

 The first author acknowledges the support of FCT - Funda\c{c}\~ao para a Ci\^{e}ncia e Tecnologia under the projects: UIDP/04561/2020 and UIDP/04674/2020, and the research centers CMAFcIO -- Centro de Matem\'{a}tica, Aplica\c{c}\~{o}es Fundamentais e Investiga\c{c}\~{a}o Operacional and CIMA -- Centro de Investigação em Matemática e Aplicações. 

The authors are grateful to Imme van den Berg for valuable comments on a preliminary version of this paper.

\bibliographystyle{plain}
\bibliography{References}

\end{document}